\documentclass[12pt,a4paper]{article}
\usepackage[utf8]{inputenc}
\usepackage[T1]{fontenc}
\usepackage[slovak,english]{babel}
\usepackage{amsthm}
\usepackage{amssymb}
\usepackage{amsbsy}
\usepackage{amsmath}
\usepackage{pgf,tikz}
\usepackage{asymptote}
\usepackage{mathtools}
\usepackage{enumerate}
\usepackage{hyperref}
\usetikzlibrary{arrows}
\usepackage{tikz-cd}
\usetikzlibrary{cd}
\usepackage[all]{xy}
\input xy
\usetikzlibrary{babel}
\xyoption{all}
\usepackage{ upgreek }
\usepackage{circledsteps}
\usepackage{xfrac}
\renewcommand{\*}{~\Circled[inner ysep=0.05pt, inner xsep=0.05pt]{*}~}
\usepackage{mathrsfs}
\newcommand{\notimplies}{%
  \mathrel{{\ooalign{\hidewidth$\not\phantom{=}$\hidewidth\cr$\implies$}}}}

\newtheorem{theorem}{Theorem}[section]
\newtheorem{lemma}[theorem]{Lemma}
\newtheorem{corollary}[theorem]{Corollary}
\theoremstyle{definition}
\newtheorem{remark}[theorem]{Remark}
\newtheorem{definition}[theorem]{Definition}
\newtheorem{proposition}[theorem]{Proposition}

\usepackage{url}

\newcommand{\Aut}[1]{\text{Aut}(#1)}

\begin{document}

\begin{center}

\Large
\textbf{Combinatorial spectra of graphs}

\end{center}

\begin{center}
\large
Martin Dz\'urik
\end{center}

\begin{center}
\textit{Department of Mathematics and statistics, Faculty of Science,
Masaryk University, Kotl\'a\v rsk\'a 2, CZ-61137 Brno, Czech Republic}
\end{center}

\begin{center}

\end{center}
\begin{center}

\end{center}

\begin{center}
\textbf{Abstract}
\end{center}
In this article we are introducing combinatorial spectra of graphs, this is a generalization of $H$-Hamiltonian spectra. The main motivation was to made from $H$-Hamiltonian spectra an operation
and develop some algebra in this field. An improved version of this operation form a commutative monoid. The most important thing is that most of the basic concepts of graph theory, such as maximum pairing, vertex and edge connectivity and coloring, Ramsey numbers, isomorphisms and regularity, can be expressed in the language of this operation.

%This approach has led us to one more generalization, which we call the combinatorial spectrum. This spectrum is now for $R$-weighted graphs, more precisely for sets of $R$-weighted graphs and we denote it $\mathscr{H} * \mathscr{G}$. This will play a role of multiplication, we also get some addition $+$. When we denote $\mathscr{P}(Graphs)$ the set of all sets of graphs we will get that $(\mathscr{P}(Graphs),*,+)$ is a $R$-module, semigroup (or comutative monoid) and something close to ring. The most important thing is that most of the basic concepts of graph theory, such as maximum pairing, vertex and edge connectivity and coloring, Ramsey numbers, isomorphisms and regularity, can be expressed in the language of these combinatorial spectra. Together with the already mentioned calculus, it gives us a hope that all these concepts could be linked, studied together and transfer the results from one to the other.

\newpage

\begin{figure}
  \centering
  \includegraphics[width=14cm]{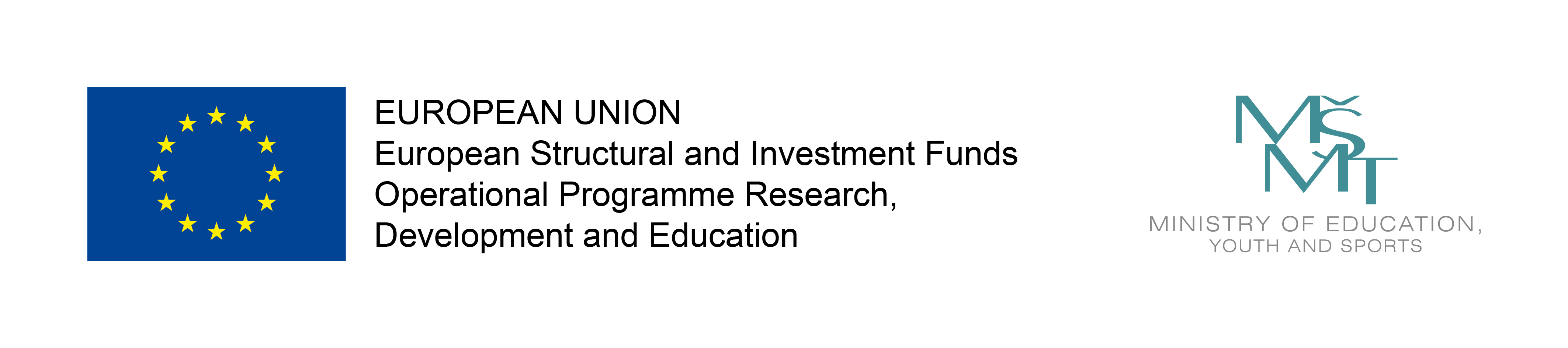}
  \label{fig:test}
\end{figure}
This work was supported from Operational Programme Research, Development
and Education – „Project Internal Grant Agency of Masaryk University" (No.
CZ.02.2.69/0.0/0.0/19\_073/0016943)

\section{Introduction}

In the article \cite{cit1} we have created a generalization of the Hamiltonian spectrum of a graph $G$ called the $H$-Hamiltonian spectrum of the graph $G$ denoted by ${\mathscr{H}}_{H}(G)$. Not only does this generalization give us the opportunity to talk, for example, about the isomorphism of graphs and the regularity of graphs in the language of these spectra, but there are several relations between ${\mathscr{H}}_{H}(G)$ and ${\mathscr{H}}_{H'}(G)$ for related $H$ and $H'$, for example for $H'=\overline{H}$. And so this brings some basic calculus to this area. 

This approach has led us to one more generalization, which we call the combinatorial spectrum. This spectrum is now for $R$-weighted graphs, more precisely for sets of $R$-weighted graphs and we denote it $\mathscr{H} * \mathscr{G}$. This will play a role of multiplication, we also get some addition $+$. When we denote $\mathscr{P}(Graphs)$ the set of all sets of graphs we will get that $(\mathscr{P}(Graphs),*,+)$ is a $R$-module, semigroup (or commutative monoid) and something close to ring. The most important thing is that most of the basic concepts of graph theory, such as maximum pairing, vertex and edge connectivity and coloring, Ramsey numbers, isomorphisms and regularity, can be expressed in the language of these combinatorial spectra. Together with the already mentioned calculus, it gives us a hope that all these concepts could be linked, studied together and transfer the results from one to the other.

In this article we will recall a definition of generalized Hamiltonian spectra of
undirected finite graphs from the article \cite{cit1} and we will define combinatorial spectra. Then we will show some basic 
algebraic properties of these spectra and describe connections to another 
topics in graph theory. At the end we will show another properties of combinatorial spectra.

The author would like to thank Lukáš Vokřínek for many helpful discussions.

Recall that,
a graph $G$ is {a} pair 
\[
G=(V(G),E(G)),
\]
where $V(G)$ is {a} finite set of vertices of $G$ and $E(G)\subseteq V(G) \times V(G)$,
a {symmetric} antireflexive relation, {is a set of edges}. {We will denote an edge between $v$ and $u$ by $\{v,u\}$.}

Let us recall the definition of the generalized Hamiltonian spectrum from \cite{cit1}.
\begin{definition}[\cite{cit1}]\label{dstara}
Let $G,H$ {be} graphs such that $|V(G)|=|V(H)|$ and \\ $f\,:\,V(H)\rightarrow V(G)$ {is a}
bijection, then we call $f$ {a} {\it pseudoordering} on {the} graph $G$ (by $H$), 
denote 
\[ 
s_H(f,G)=\sum_{\{x,y\}\in E(H)}\rho_G(f(x),f(y)),
\] 
where
$\rho_G (x,y)$ is {the} distance of $x,y$ in {the} graph $G$. Then 
\[
\mathscr{H}_{H}(G)=\{s_H(f,G)|f~\text{pseudoordering on}~ G \text{ by $H$}\}
\]
is {the} {\it $H$-Hamiltonian spectrum} of {the} graph $G$.
\end{definition}

Now we will define a combinatorial spectrum.
In the rest of this article we will use a notation $R$ for a ring.
In combinatorial spectra we talk about $R$-weighted complete graphs,
for many application we will use simply $R=\mathbb{R}$.

\begin{definition}
Let $H$ and $G$ be $R$-weighted complete graphs such that $|V(G)|=|V(H)|$ with weighted functions 
\[
v_H:E(H) \rightarrow R,
\] 
\[
v_G:E(G) \rightarrow R,
\]
and $f\,:\,V(H)\rightarrow V(G)$ be a bijection then we define $H *_f G$ as follows 
\[
V(H *_f G) = V(H), \quad E(H *_f G)=E(H), \quad v_{ H *_f G}(e)=v_H(e) \cdot v_G(f(e))
\]
\end{definition}

\begin{definition}\label{dprva}
Let $H$ and $G$ be $R$-weighted complete graphs such that $|V(G)|=|V(H)|$, we define
\[
H*G=\{H *_f G ~| ~f : V(G) \rightarrow V(H),\text{ bijection}\}
\]
and
\[
H \* G = \sfrac{\langle H*G \rangle}{iso}, 
\]
where $\langle-\rangle$ means closure under isomorphisms (but on given set of vertices).
\end{definition}

\begin{remark}
$\*$ clould be equivalently defined by 
\[
H*G=\{[H *_f G] ~| ~f : V(G) \rightarrow V(H),\text{ bijection}\},
\]
where $[K]$ is an isomorphism class of graph $K$.
\end{remark}

\begin{remark}
\[
\langle-\rangle : \mathscr{P}(R-Graphs_n) \rightarrow \mathscr{P}(R-Graphs_n)
\]
is an isotonic idempotent map.

\end{remark}

\begin{definition}
Let $\mathscr{H},~\mathscr{G}$ be sets (or classes) of $R$-weighted complete graphs with the same number of vertices than we define 

\[
\mathscr{H}*\mathscr{G}=
\bigcup_{\substack{G \in \mathscr{G}\\
H \in \mathscr{H}}} 
H*G,\quad
\mathscr{H}\*\mathscr{G}=
\bigcup_{\substack{G \in \mathscr{G}\\
H \in \mathscr{H}}} 
H\*G.
\]
\end{definition}

\begin{remark}
Now $*$ and$\*$are operations
\[\begin{tikzcd}[row sep=tiny]
*:\mathscr{P}(R-Graphs_n)\times \mathscr{P}(R-Graphs_n)\arrow[r] &\mathscr{P}(R-Graphs_n)
\\
\*:\sfrac{\langle\mathscr{P}(R-Graphs_n)\rangle}{iso}\times \sfrac{\langle\mathscr{P}(R-Graphs_n)\rangle}{iso}\arrow[r] &\sfrac{\langle\mathscr{P}(R-Graphs_n)\rangle}{iso}.
\end{tikzcd}\]
Where $\mathscr{P}(R-Graphs_n)$ denotes the set of all sets of $R$-weighted complete graphs on $n$ vertices.

The first part is obvious and the second part is given by the following commutative diagram. 
\[
\begin{tikzcd}
\mathscr{P}(R-Graphs_n)\times \mathscr{P}(R-Graphs_n)  \arrow[r, "*"] \arrow[d, two heads,"\langle - \rangle \times\langle - \rangle"', bend right] & \mathscr{P}(R-Graphs_n)\arrow[d, two heads,"\langle - \rangle"', bend right]\\
\langle\mathscr{P}(R-Graphs_n)\rangle\times \langle\mathscr{P}(R-Graphs_n)\rangle \arrow[u, bend right,hook] \arrow[d, "\cong"]\arrow[r, dotted,"!"] & \langle\mathscr{P}(R-Graphs_n)\rangle\arrow[d, "\cong"]\arrow[u, bend right,hook]\\
\sfrac{\langle\mathscr{P}(R-Graphs_n)\rangle}{iso}\times \sfrac{\langle\mathscr{P}(R-Graphs_n)\rangle}{iso} \arrow[r, dotted,"!"] & \sfrac{\langle\mathscr{P}(R-Graphs_n)\rangle }{iso},
\end{tikzcd}
\]
where the bottom arrow is $\*$. Where 
\[\langle\mathscr{P}(R-Graphs_n)\rangle
=\{ \langle\mathscr{G}\rangle| \mathscr{G} \in \mathscr{P}(R-Graphs_n)\}
\]

The middle arrow exists uniquely, because of the fact \[
\langle\mathscr{G}*\mathscr{H}\rangle=\langle\mathscr{G}\rangle * \langle\mathscr{H}\rangle.
\]
\end{remark}

\begin{definition}\label{dirho}
Let $G$ be an (connected) unweighted graph then we define $R$-weighted complete graphs $I(G)$, $\rho(G)$ as follows,
\[
V(I(G))=V(\rho(G))=V(G),
\]
\[
v_{I(G)}(e)=
\begin{cases} 
1  & \text{for }e \in E(G)\\
0  &\text{for }e\notin E(G)
\end{cases},
\quad v_{\rho(G)}(\{x,y\}) = \rho_G(x,y).
\]
\end{definition}

In the rest of the section we will show some algebraic properties $*$.

\begin{lemma}\label{l2}
Let $F$, $G$ and $H$ be $R$-weighted complete graphs such that $|V(F)|=|V(G)|=|V(H)|$ and $f\,: V(F) \rightarrow V(G)$, $f\,: V(G) \rightarrow V(H)$ be pseudoorderings then 
\[
F *_{f}( G*_{g } H)=(F*_f G)*_{g \circ f} H , \quad
F \*_{f}( G\*_{g } H)=(F\*_f G)\*_{g \circ f} H .
\]

\end{lemma}

\begin{proof}
\[
v_{(F*_f G)*_{g\circ f} H}(e)=v_{F*_f G}(e) \cdot v_H(g(f(e)))= v_F(e) \cdot v_G(f(e)) \cdot v_H(g(f(e))),
\]
\[
v_{F *_{f}( G*_{g} H)}(e)= v_{F}(e) \cdot v_{G*_{g} H}(f(e))= 
v_{F}(e) \cdot v_G (f(e)) \cdot v_H (g(f (e))).
\]
For $\*$ similar.
\end{proof}

\begin{lemma}\label{l1}
Let $G$ and $H$ be $R$-weighted complete graphs such that $|V(G)|=|V(H)|$ and $f\,: V(F) \rightarrow V(G)$ be pseudoordering then there is an isomorphism
\[
\xymatrix{&f:~H*_f G & G*_{f^{-1}} H \ar@<1ex>[l]
\ar@<1ex>[l];[]~:f^{-1}.}
\]
\end{lemma}

\begin{proof}
\[
v_{G*_{f^{-1}} H}(f(e))=v_G (f(e)) \cdot v_H ((f^{-1}\circ f)(e))=v_H(e) \cdot v_G(f(e))= v_{H*_f G}(e).
\]
\end{proof}

\begin{lemma}\label{l3}
Let $H$ be a $R$-weighted complete graph then 
\[
H*_f I(K_{|V(H)|})= H
\]
\end{lemma}
\begin{proof}
\[
v_{H*_f I(K_{|V(H)|})}(e)=v_H(e) \cdot 1 = v_H(e).
\]
\end{proof}
\begin{theorem}

Let $\mathscr{F},~\mathscr{G},~\mathscr{H}$ be sets of $R$-weighted complete graphs of the same number of vertices then
\[
(\mathscr{F}*\mathscr{G})*\mathscr{H}=\mathscr{F}*(\mathscr{G}*\mathscr{H}), \quad
(\mathscr{F}\*\mathscr{G})\*\mathscr{H}=\mathscr{F}\*(\mathscr{G}\*\mathscr{H}).
\]
\[
\mathscr{H}\*\mathscr{G} = \mathscr{G}\*\mathscr{H}, \quad \mathscr{H}\*\bigl\{I(K_{|V(H)|})\bigr\}=\mathscr{H}.
\]

\end{theorem}

\begin{proof}
Equalities follows from Lemmas \ref{l2}, \ref{l1} and \ref{l3}. For example for first one
\begin{alignat*}{3}
(\mathscr{F}&*\mathscr{G})*\mathscr{H} & &\\
&=\bigl\{ (F*_f G)*_g H  &\bigm |     & F\in\mathscr{F},G\in\mathscr{G},H\in\mathscr{H},f,g  \text{ pseudoorderings}     \bigr\} \\
&=\bigl\{ (F*_f G)*_{g'\circ f} H &\bigm |    &  F\in\mathscr{F},G\in\mathscr{G},H\in\mathscr{H},f,g'  \text{ pseudoorderings}     \bigr\}\\
&=\bigl\{ F *_{f}( G*_{g' } H)       &\bigm |     & F\in\mathscr{F},G\in\mathscr{G},H\in\mathscr{H},f,g'  \text{ pseudoorderings}     \bigr\}\\
&=\mathscr{F}*(\mathscr{G}*\mathscr{H}). & & 
\end{alignat*}
For another similar.
\end{proof}

\begin{corollary}
$
(\mathscr{P}(R-Graphs_n),*)
$ is a semigroup and \[(\sfrac{\langle\mathscr{P}(R-Graphs_n)\rangle}{iso},\*,\bigl\{I(K_{|V(H)|})\bigr\})\] is a comutative monoid.
\end{corollary}

\begin{definition}
Let $G$ be a $R$-weighted complete graph with weighted function $v_H$ then we define 
\[
s(G)=\sum_{e \in E(G)} v_H(e).
\] 
And for $\mathscr{G} \in \mathscr{P}(R-Graphs_n)$ we define 
\[
s(\mathscr{G})=\{s(G)|G\in\mathscr{G} \}.
\]
\end{definition}

%distributivita

\begin{definition}
Let $H_1,H_2\in R-Graphs_n$ such that $V(H_1)=V(H_2)$ then we define an operation $+$ on $ R-Graphs_n$ as follows
\[
V(H_1+H_2)=V(H_1), \quad 
v_{H_1+H_2}(e)=v_{H_1}(e)+v_{H_2}(e).
\]
We also define an operation on $ \mathscr{P}(R-Graphs_n)$ as follows, let $\mathscr{H}_1,\mathscr{H}_2\in \mathscr{P}(R-Graphs_n)$ such that all graphs have the same set of vertices then  
\[
\mathscr{H}_1+\mathscr{H}_2=\left\{ H_1+H_2\middle|H_1\in \mathscr{H}_1,  H_2\in \mathscr{H}_2              
\right\}
\]
\end{definition}

\begin{lemma}\label{l8}
Let $H_1,H_2,G$ be $R$-weighted complete graphs with the same number of vertices such that $V(H_1)=V(H_2)$ then it holds
\[
(H_1+H_2)*G\subseteq H_1 * G + H_2 * G
\]
\[
(H_1+H_2)\*G\subseteq H_1 \* G + H_2 \* G. 
\]
\end{lemma}

\begin{proof}
It is easy to prove that
\[
(H_1+H_2)*_f G= H_1 *_f G + H_2 *_f G
\]
and from this we get the statement of the lemma.

For $\*$ similar.
\end{proof}

\begin{lemma}\label{l-1}
Let $H_1,H_2,G$ be $R$-weighted complete graphs with the same number of vertices such that $V(H_1)=V(H_2)$
and $|s(H_1*G)|=1$ then it holds 
\[
s((H_1+H_2)*G) = s(H_1 * G) + s(H_2 * G)
\]
\end{lemma}

\begin{proof}
From the previous lemma we get one inclusion, lets take an arbitrary element of $ s(H_1 * G) + s(H_2 * G)$ which is equal to 
\[
s(H_1 *_f G) + s(H_2 *_g G)\overset{|s(H_1*G)|=1}{=} s(H_1 *_g G) + s(H_2 *_g G)= s((H_1 + H_2) *_g G)
\] 
and the last thing is an element of $s((H_1+H_2)*G)$.
\end{proof}

\begin{theorem}\label{Tad}
Let $\mathscr{H}_1, \mathscr{H}_2 ,~\mathscr{G}\in \mathscr{P}(R-Graphs_n)$ then
\[
(\mathscr{H}_1+\mathscr{H}_2)*\mathscr{G}\subseteq \mathscr{H}_1 * \mathscr{G} + \mathscr{H}_2 * \mathscr{G} 
\]
\[
(\mathscr{H}_1+\mathscr{H}_2)\*\mathscr{G}\subseteq \mathscr{H}_1 \* \mathscr{G} + \mathscr{H}_2 \* \mathscr{G} 
\]
\end{theorem}
\begin{proof}
\begin{align*}
(&\mathscr{H}_1+\mathscr{H}_2)*\mathscr{G} \\
&= \bigcup_{\substack{H_1 \in \mathscr{H}_1,H_2 \in \mathscr{H}_2,\\ G \in \mathscr{G}}} (H_1+H_2)*G \overset{L\ref{l8}}{ \subseteq }
\bigcup_{\substack{H_1 \in \mathscr{H}_1,H_2 \in \mathscr{H}_2,\\ G \in \mathscr{G}}} H_1 *  G + H_2 * G\\
 &= \mathscr{H}_1 * \mathscr{G} + \mathscr{H}_2 * \mathscr{G} .
\end{align*}
For $\*$ similar.
\end{proof}

\begin{theorem}\label{diststrict}
Let $\mathscr{H}_1,~\mathscr{G}\in \mathscr{P}(R-Graphs_n)$ such that for all $H_1 \in \mathscr{H}_1,~ G \in \mathscr{G}$ holds $|s(\{H_1\}*\{G\})|=1$ then
\[
s((\mathscr{H}_1+\mathscr{H}_2)*\mathscr{G})= s(\mathscr{H}_1 * \mathscr{G}) + s(\mathscr{H}_2 * \mathscr{G})= s(\mathscr{H}_1 * \mathscr{G} +\mathscr{H}_2 * \mathscr{G}).
\]
\end{theorem}

\begin{proof}
The proof is similar as before, but we replace the lemma \ref{l8} by \ref{l-1}, where is equality. 
\end{proof}

\section{Connection to another topics}
In this section we are going to show how can be many standard concepts in graph theory be expressed in the language of combinatorial spectra.

\begin{theorem}\label{hamspec}
Let $H$ and $G$ be unweighted graphs such that $|V(G)|=|V(H)|$ then
\[
\mathscr{H}_{H}(G)=s(\{I(H)\}*\{\rho(G)\}).
\]
\end{theorem}

\begin{proof}
The proof is just a straight forward comparison of definitions \ref{dstara}, \ref{dprva} and \ref{dirho}.
\end{proof}

\begin{proposition}
Let $\mathscr{G}$ be a set of $R$-weighted complete graphs with the same number of vertices then 
\[
\langle \mathscr{G} \rangle = K_{n}* \mathscr{G},
\]  
where $\langle - \rangle$ is closure under isomorphisms introduced in definition \ref{dprva} and $n$ is a number of vertices of graphs in $\mathscr{G}$.
\end{proposition}

\begin{proof}
proof follows from this equation
\[
K_{n}*_f G= f^{-1}(G).
\]
\end{proof}

\begin{theorem}
Let $G$ be an unweighted graph then 
\[
s(\mathscr{P}_{|V(G)|}* I(G))=\{ \text{cardinalities of pairings in the graph }G\},
\]
where \begin{align*}
\mathscr{P}_{|V(G)|} &= \left\{I(P_{|V(G)|,k})\middle| k \in \left\{0, \dots ,   \left\lfloor\frac{|V(G)|}{2}\right\rfloor  \right\}   \right\},\\
P_{|V(G)|,k}&=
 \bigsqcup_{i=1}^{k} \bigl(\xymatrix{ \cdot \ar@{-}[r] &\cdot  } \bigr) \sqcup \bigsqcup_{i=1}^{|V(G)|-2k} \cdot 
\end{align*}
\end{theorem}

\begin{proof}
Let $\{e_1,\dots,e_k\}$ be a pairing in the graph $G$, then we can construct a pseudoordering $ f : P_{|V(G)|,k} \rightarrow V(G)$ such that each edge of $P_{|V(G)|,k}$ maps to $e_i$ for some $i$. For this $f$ we have 
\[
s(I(\mathscr{P}_{|V(G)|})*_f I(G))=k,
\]
hence 
\[
s(\mathscr{P}_{|V(G)|}* I(G))\supseteq\{ \text{cardinalities of pairings in the graph }G\}.
\]
Let $ f : P_{|V(G)|,k} \rightarrow V(G)$ be a pseudoordering, when we delete from $P_{|V(G)|,k}$ edges which are mapped to no-edges of $G$ we will get a pseudoordering 
$ g : P_{|V(G)|,l} \rightarrow V(G)$ such that 
\[
s(I(P_{|V(G)|,k})*_f I(G))= s(I(P_{|V(G)|,l})*_g I(G)).
\]
But now image of $g$ form a pairing in the graph $G$ and hence
\[
s(\mathscr{P}_{|V(G)|}* I(G))\subseteq\{ \text{cardinalities of pairings in the graph }G\}.
\]
\end{proof}

\begin{theorem}\label{thmdeg}
\[
s\left(S_{|V(G)|}*I(G)\right)=\{deg_G(v)| v\in V(G)\},
\]
where $S_{|V(G)|}=K_{1,|V(G)|-1}$ is a star with same number of vertices as the graph $G$ and $deg_G(v)$ is a degree of vertex $v$ in the graph $G$.
\end{theorem}

\begin{proof}
Let $f:S_{|V(G)|} \rightarrow I(G)$ be a pseudoordering, then 
\[
s\left(S_{|V(G)|}*_f I(G)\right)= deg_G(f(c)),
\]
where $c$ is a center of the star. Indeed, let $\{u,v\}\in E\left(S_{|V(G)|}*_f I(G)\right)$ then weight of this edge is zero if both vertices $u,v$ are different from vertex $c$.
And weight of edge \{c,v\} is
\begin{align*}
v_{S_{|V(G)|}*_f I(G)} \{c,v\} = \begin{cases}1 & \text{if }\{f(c),f(v)\}\in E(G) \\
0 & \text{if }\{f(c),f(v)\}\notin E(G).
\end{cases}
\end{align*}
\end{proof}

\begin{theorem}
$G$ is vertex $k$-colorable if and only if 
\[
\min s\left( \mathscr{H}_k * \{I(G)\}\right)= 0,
\]
where
\[
\mathscr{H}_k =\left\{    I\left(\bigsqcup_{i=1}^{k} K_{n_i}\right) \middle| n_1,\dots,n_k \in \mathbb{N}, \sum_{j=0}^k n_j = |V(G)|   \right\}. 
\]
\end{theorem}

\begin{proof}
Let we have a proper $k$-coloring $g: V(G) \rightarrow \{1,\dots,k\}$ of the graph $G$, let $n_i=|g^{-1}(i)|$ be a number of vertices of color $i$. Take an arbitrary bijection 
\[
f_i:V(K_{n_i}) \rightarrow g^{-1}(i),
\]
for all $i \in \{1,\dots,k\}$.
Then we define a bijection
\[
f=\bigsqcup_{i=1}^{k}f_i:\bigsqcup_{i=1}^{k} K_{n_i} \rightarrow   V(G),
\]
this pseudoordering satisfy 
\[
s\left( I\left(\bigsqcup_{i=1}^{k} K_{n_i}\right) *_f I(G)\right)= 0.
\]
Indeed, let $\{v,u\} \in E(\bigsqcup_{i=1}^{k} K_{n_i} )$, then $f(v),f(u)$ have the same color in coloring $g$ hence there is not an edge $\{f(v),f(u)\}$, this means that all nonzero edges in 
$I(\bigsqcup_{i=1}^{k} K_{n_i})$ maps to zero by $f$. 

Let now 
\[
\min s\left( \mathscr{H}_k * \{I(G)\}\right)= 0,
\]
this means that there exists a pseudoordering $f$ and a decomposition $\sum_{j=0}^k n_j = |V(G)|$ such that 
\[
s\left( I\left(\bigsqcup_{i=1}^{k} K_{n_i}\right) *_f I(G)\right)= 0.
\]
From this we can define a coloring similar as in the opposite direction.  
\end{proof}

\begin{definition}
Let $a \in R$ then we define 
\[
U_a: R-Graph_n \rightarrow Graph_n,
\]
defined as follows, for graph $G$ with weight function $v_G$, we define unweighted graph $U_a(G)$,
\[
V(U_a(G))=V(G),~E(U_a(G))=\left\{e\in E(G) \middle|  v_G (e)=a   \right\}.
\]

\end{definition}

Following lemmas will be used for description of edge colorability and edge connectivity.

\begin{lemma}\label{l4}
Let $G\neq K_{|V(G)|}$ be a unweighted graph then 
\[
I(G)*I(K_{|V(G)|}\setminus e)
\]
is equal to $I$ of a set of all graphs obtainable from $G$ by deleting at most 1 edge.
\end{lemma}

\begin{proof}
In a pseudoordering is some edge $e'$ of $G$ mapped to the edge $e$ and this pseudoordering gives us a graph $G \setminus e'$.   
\end{proof}

\begin{lemma}\label{l5}
Let $G\neq K_{|V(G)|}$ be a unweighted graph then 
\[
I(G)*I(K_{|V(G)|}\setminus e)^{*k}
\]
is equal to $I$ of set of all graphs obtainable from $G$ by deleting at most $k$ edges. Where 
\[
H^{*k}=\underbrace{H*\dots *H}_k
.\]
\end{lemma}

\begin{proof}
Proof is given by iteration of the previous lemma (\ref{l4}).
\begin{align*}
I(G)*I(K_{|V(G)|}\setminus e)^{*k} &= I(G)*I(K_{|V(G)|}\setminus e)^{*(k-1)}*I(K_{|V(G)|}\setminus e)\\
&=\{G \setminus \leq k -1 \text{ edges}\}*I(K_{|V(G)|}\setminus e)\\
&=\{G \setminus \leq k  \text{ edges}\}.
\end{align*}
\end{proof}

\begin{definition}
Let $\mathscr{G} \in \mathscr{P}(R-Graphs_n)$
if there exists $n \in \mathbb{N}$ such that for all $n\leq k\in \mathbb{N}$
\[
\mathscr{G}^{*k}=\mathscr{G}^{*n},
\]
we will denote 
\[
\mathscr{G}^{*\infty}=\mathscr{G}^{*n}.
\]
This is a standard notation from the theory of semigroups.
Also we will define for $G \in R-Graphs_n$
\[
G^{*\infty}=\{G\}^{*\infty}.
\]
\end{definition}

\begin{lemma}\label{l6}
\[
I(K_{|V(G)|}\setminus e)^{*\infty} \cup \left\{I(K_{|V(G)|})\right\}
\]
is equal to $I$ of set of all graphs with given number of vertices.
\end{lemma}

\begin{proof}
\[
I(K_{|V(G)|}\setminus e)^{*\infty}=I(K_{|V(G)|}\setminus e)*I(K_{|V(G)|}\setminus e)^{*\infty},
\]
hence by lemma \ref{l5} it is equal to $I$ of set of all graphs obtainable from $K_{|V(G)|}\setminus e$ by deleting at most $\infty$ edges. Then it is any graph up to $I(K_{|V(G)|})$.
\end{proof}

\begin{lemma}\label{l7}
\[I(G)*\left(\sum_{i=1}^{k-1} \left(I(K_{|V(G)|}\setminus e)^{*\infty} \cup \left\{I(K_{|V(G)|})\right\}\right)+ I(K_{|V(G)|}) \right)\]
is equal to a set of all $k$-colorings (possibly non-proper) of the graph $G$ where we add zeros to non-edges of $G$. 
\end{lemma}

\begin{proof}
\begin{align*}
\sum_{i=1}^{k-1} &\left(I(K_{|V(G)|}\setminus e)^{*\infty} \cup \left\{I(K_{|V(G)|})\right\}\right)\\
\overset{L\ref{l6}}{=}&\sum_{i=1}^{k-1}\left( \text{set of all graphs weighted by 0 and 1}              \right)\\
\overset{\phantom{L\ref{l6}}}{=} &(\text{set of all graphs weighted by }0,\dots,k-1),
\end{align*}
hence 
\begin{align*}
&I(G)*\left(\sum_{i=1}^{k-1} \left(I(K_{|V(G)|}\setminus e)^{*\infty} \cup \left\{I(K_{|V(G)|})\right\}\right)+ I(K_{|V(G)|}) \right)\\
=&I(G)*\left((\text{set of all graphs weighted by }0,\dots,k-1)+I(K_{|V(G)|}) \right)\\
=&I(G)*\phantom{(}(\text{set of all graphs weighted by }1,\dots,k)\\
=&(\text{set of all }k\text{-colorings of the graph } G).
\end{align*}
\end{proof}

\begin{theorem}
$G$ is edge-$k$ colorable if and only if
\[
\exists H \in I(G)*\sum_{i=1}^{k-1} \left( \left(I(K_{|V(G)|}\setminus e)^{*\infty} \cup \left\{I(K_{|V(G)|})\right\}\right) + I(K_{|V(G)|}) \right)
\]
such that
\[
\forall a \in \{1,\dots,k\} :~ s(I(U_a(H))*S_{|V(G)|})\leq 1.
\]
\end{theorem}

\begin{proof}
By lemma \ref{l7} we have that condition 
\[
\exists H \in I(G)*\sum_{i=1}^{k-1} \left( \left(I(K_{|V(G)|}\setminus e)^{*\infty} \cup \left\{I(K_{|V(G)|})\right\}\right) + I(K_{|V(G)|}) \right),
\]
means there exists a coloring of the graph $G$. By the theorem \ref{thmdeg}
\[
 s(I(U_a(H))*S_{|V(G)|})\leq 1,
 \]
is equivalent to the fact that degree of any vertex of $U_a(H)$ is at most $1$ for $a\in \{1,\dots,k\}$. Hence second condition says that the coloring is proper.
\end{proof}

\begin{theorem}
$G$ has friendly bisection if and olny if 
\[
\exists H \in   \left\{   I\left(K_{|V(G)|} \right)- 2\cdot I\left(K_{\frac{|V(G)|}{2},\frac{|V(G)|}{2}} \right)    \right\} * \{I(G)\}   
\]
such that
\[
s\left(H * S_{|V(G)|}\right)\geq 0.
\]
\end{theorem}

\begin{proof}
\[
\left\{   I\left(K_{|V(G)|} \right)- 2\cdot I\left(K_{\frac{|V(G)|}{2},\frac{|V(G)|}{2}} \right)    \right\} * \{I(G)\}
\]
is the set of all weighted graphs obtainable from $I(G)$ by given process, let $A\subseteq V(G)$ be an arbitrary subset. If we now denote $I(G)_A$ a graph given by changing all signs of edges of $I(G)$ between set $A$ and $V(G) \setminus A$. The last combinaorial spectra is equal to 
\[
\left\{I(G)_A\middle||A|=\frac{|V(G)|}{2}\right\}.
\] 
Let $A\subseteq V(G)$ be a subset and we denote $B=A\subseteq V(G)$. Now we will compute a combinatorial spectrum $s\left( I(G)_A* S_{|V(G)|}\right)$. 
\[
s\left( I(G)_A*_f S_{|V(G)|}\right)= \begin{cases}
|N(f^{-1}(c)) \cap A| - |N(f^{-1}(c)) \cap B| & \text{if } f^{-1}(c) \in A\\
|N(f^{-1}(c)) \cap B| - |N(f^{-1}(c)) \cap A| & \text{if } f^{-1}(c) \in B,
\end{cases}
\]
where $N(v)\subseteq V(G)$ is the neighborhood of the vertex $v$ and $c$ is a center of the star $S_{|V(G)|}$.

Let now $H \in \left\{I(G)_A\middle||A|=\frac{|V(G)|}{2}\right\}$, such that  
\[
s\left(H * S_{|V(G)|}\right)\geq 0.
\]
By previous result it is equivalent to $A,V(G) \setminus A$ being a friendly bisection of $G$.
\end{proof}

\begin{lemma}\label{l9}
Let $G$ be unweighted graph, then $
G$ is edge connected if and only if 
\[
\min \left(s\left( \mathscr{H}_{|V(G)|}* I(G)\right)\right)\geq 1,
\]
where
\[
\mathscr{H}_{|V(G)|}=\left\{I(K_{n,|V(G)|-n}) ~\middle|~ 1\leq n \leq|V(G)|-1 \right\}.
\]
\end{lemma}

\begin{proof}
The proof follows from the fact that 
\[
s(K_{n,|V(G)|-n}*_f I(G))
\] 
is equal to the number of edges of $G$ between sets of vertices $f(A),f(B)\subseteq V(G)$, where $A,B\subseteq V(K_{n,|V(G)|-n})$ are two sides of this bipartite graph. 
And from the fact that a $G$ is connected if and only if for any partition $P,V(G)\setminus P\subseteq V(G)$ of the graph $G$, there exists an edge in $G$ connecting $P$ and $V(G)\setminus P$.
\end{proof}

\begin{theorem}
Let $G$ be unweighted graph, then
$
G \text{ is edge } (k+1) \text{-connected iff } 
$
\[
\min \left(s\left( \mathscr{H}_{|V(G)|}* I(K_{|V(G)|}\setminus e)^{*k}*I(G)\right)\right)\geq 1,
\]
where
\[
\mathscr{H}_{|V(G)|}=\left\{I(K_{n,|V(G)|-n}) ~\middle|~ 1\leq n \leq|V(G)|-1 \right\}.
\]
\end{theorem}

\begin{proof}
The proof follows from lemmas \ref{l5} and \ref{l9}. By lemma \ref{l5} is the spectrum
\[
\left(I(K_{|V(G)|}\setminus e)^{*k}*I(G)\right)
\]
equal to $I$ of set of all graphs obtainable from $G$ by deleting at most $k$ edges. Now by lemma \ref{l9} the whole formula is equivalent to fact 
that all graphs obtainable from $G$ by deleting at most $k$ edges are connected.
\end{proof}

Following lemmas will be used for description of vertex connectivity.

\begin{lemma}\label{l10}
Let $H$ be $R$-weighted complete graph, then 
\[
H*{I(\overline{K_{1,|V(G)|-1}})}^{*k}
\]
is a set of all graphs obtainable from $H$ by vanishing all edges neighboring to at most $k$ vertices. By this we mean a set \[
\left\{ H_{\setminus A}\middle|A \subseteq V(H),~|A|\leq k\right\},
\]
where $H_{\setminus A}$ is a graph given by vanishing all edges which have at least one endpoint in $A$. 
And for $G$ unweighted graph we denote
\[
\overline{G}=K_{|V(G)|}\setminus G.
\]
It is a complement of a graph $G$.
\end{lemma}

\begin{proof}
The proof is similar to proof of lemma \ref{l5}.
\end{proof}

\begin{lemma}\label{l11}
Let $G$ be unweighted graph, then
\[
\left(I(G)+ \left(-\frac{1}{2}\right) \cdot K_{|V(G)|} \right) * I\left(\overline {K_{1,|V(G)|-1}}\right)^{*k} + \frac{1}{2} \cdot K_{|V(G)|}
\]
is a set of all graphs obtainable from $I(G)$ by changing all edges neighboring to at most $k$ vertices to $\frac{1}{2}$.
\end{lemma}

\begin{proof}
Follows from lemma \ref{l10}
When we have some edge which is not a neighbor of any of those vertices then it's weight will be changed by $-\frac{1}{2}$ then times $1$ and $+\frac{1}{2}$, hence it will be the same.
But if we have a vertex neighboring to some of those vertices then it's weight will be changed by $-\frac{1}{2}$ then times $0$ and $+\frac{1}{2}$, hence it will be equal to $\frac{1}{2}$.
\end{proof}

\begin{theorem}
Let $G$ be unweighted graph such that $G\neq K_{|V(G)|}$ and $k \in \mathbb{N}$, then
$G$ is vertex $ (k+1) $-connected if and only if for all  
\[\begin{aligned}
&H \in \\
&\left(\left(I(G)+ \left(-\frac{1}{2}\right) \cdot K_{|V(G)|} \right) * I\left(\overline {K_{1,|V(G)|-1}}\right)^{*k} + \frac{1}{2} \cdot K_{|V(G)|}\right)* \mathscr{H}_{|V(G)|},
\end{aligned}\]
there exists $e\in E(H)$ such that $v_H(e)=1$ or 
\[
H \cong \frac{1}{2} \cdot I(K_{n,|V(G)|-n})
\] 
for some $n \in \{1,\dots ,|V(G)|-1 \}$.
Where
\[
\mathscr{H}_{|V(G)|}=\left\{I(K_{n,|V(G)|-n}) ~\middle|~ 1\leq n \leq|V(G)| \right\}.
\]
\end{theorem}

\begin{proof}
\begin{align*}
&\left(\left(I(G)+ \left(-\frac{1}{2}\right) \cdot K_{|V(G)|} \right) * I\left(\overline {K_{1,|V(G)|-1}}\right)^{*k} + \frac{1}{2} \cdot K_{|V(G)|}\right)* \mathscr{H}_{|V(G)|}\\
&=\left\{\substack{\text{set of all graphs obtainable from $I(G)$ by changing} \\ \text{all edges neighboring to at most $k$ vertices to $\frac{1}{2}$}} \right\}* \mathscr{H}_{|V(G)|},
\end{align*}
hence whole formula is equal to set of all graphs obtainable from $I(G)$ by changing all edges neighboring to at most $k$ vertices to $\frac{1}{2}$ 
and then leave only edges between some nonempty proper subset of vertices and it's complement. 

Let $G$ be not $(k+1)$-connected we will show that there exists a graph $H$ which does not satisfy given condition. By lemma \ref{l11}
Because $G$ is not a complete graph, know that there exist vertices $v_1,v_2,\dots, v_l \in V(G)$ where $l\leq k$ such that 
$G\setminus \{ v_1,v_2,\dots, v_l \}$ is disconnected. Let $C$ be a component of $G\setminus \{ v_1,v_2,\dots, v_l \}$, 
if we now consider a graph $H$ in the spectrum from the theorem given by changing all edges neighboring to vertices $\{ v_1,v_2,\dots, v_l \}$ to $\frac{1}{2}$ 
and then leave only edges between $C\cup \{ v_1,v_2,\dots, v_l \}$ and it's complement. By previous this graph lies in the spectrum and we will show that it not satisfy the condition.
All edges inside sets $C\cup \{ v_1,v_2,\dots, v_l \}$ and $\overline{C\cup \{ v_1,v_2,\dots, v_l \}}$ are $0$, 
all edges between $C$ and $\overline{C\cup \{ v_1,v_2,\dots, v_l \}}$ are $0$ and all edges between $ \{ v_1,v_2,\dots, v_l \}$ and $\overline{C\cup \{ v_1,v_2,\dots, v_l \}}$ are $\frac{1}{2}$. 
Hence there is no $1$ and it is not isomorphic to $\frac{1}{2} \cdot I(K_{n,|V(G)|-n})$.

For another direction let $G$ be $(k+1)$-connected we will show that all $H$ from the spectrum satisfy the condition. We already know that all $H$ from the spectrum are obtained by a set 
$\{ v_1,v_2,\dots, v_l \}\subseteq V(G)$ and a partition in to two nonempty parts $A,\overline{A}\subseteq V(G)$. If $G\setminus \{ v_1,v_2,\dots, v_l \}$ 
has a nonempty intersection with both $A$ and $\overline{A}$
then because $G\setminus \{ v_1,v_2,\dots, v_l \}$ is connected, there exists at least one edge between $A$ and $\overline{A}$, hence $H$ contains an edge weighted by $1$.

If $A\supseteq G\setminus \{ v_1,v_2,\dots, v_l \}$ (analogously for $\overline{A}$) then $\overline{A} \subseteq \{ v_1,v_2,\dots, v_l \}$ 
and this means that all edges between $A$ and $\overline{A}$
are weighted by $\frac{1}{2}$ and rest are by previous part $0$, this means that it is isomorphic to $ \frac{1}{2} \cdot I(K_{n,|V(G)|-n})$.
\end{proof}

\begin{remark}
The condition of $G$ not being a complete graph from the previous theorem can be easily said also by combinatorial spectra, it is equivalent to the following condition,
\[
s\left(I(G)*I\left( \bigl(\xymatrix{ \cdot \ar@{-}[r] &\cdot  } \bigr) \sqcup \bigsqcup_{i=1}^{|V(G)|-2} \cdot \right) \right)\neq \{1\}.
\]
\end{remark}

\begin{theorem}
Let $n>k \in \mathbb{N}$ then $n \geq R(k,k)$ if and only if 
\[
\exists G \in \bigl( 2 \cdot I(K_{n}\setminus e)^{*\infty} + (-1) \cdot K_{n}  \bigr),
\]
such that
\[
-{k \choose 2}   <
s \left( G        
*    I\left( K_k \sqcup \bigsqcup_{i=1}^{|V(G)|-k} \cdot   \right) \right)  <  {k \choose 2 }, 
\]
where $R(k,k)$ is a Ramsey number.
\end{theorem}

\begin{proof}
The idea of this condition is generate all graphs and check that there not contains a monochromatic clique. 
By lemma \ref{l6} we know that $I(K_{n}\setminus e)^{*\infty}$ is set of $I$ of all graphs of order $n$ up to a complete graph, but the complete graph is irrelevant for Ramsey condition.
Hence $\bigl( 2 \cdot I(K_{n}\setminus e)^{*\infty} + (-1) \cdot K_{n}  \bigr)$ is the set of all colorings of a graph $K_n$ by colors $1,-1$. Now the formula 
\[
s \left( G        
*    I\left( K_k \sqcup \bigsqcup_{i=1}^{|V(G)|-k} \cdot   \right) \right) 
\] 
is equal to the set of numbers of $1$ minus numbers of $-1$ in all induced subgraphs of $G$ of order of $k$.
And inequalities say that all these subgraphs are not monochromatic. 
\end{proof}

\section{Properties}

In this section we are going to show some properties of combinatorial spectra.

\begin{theorem} \label{dense}
Let $G$ be a $\mathbb{Z}$-weighted complete graph such that 
\begin{equation}\label{huste}
s\left(G*I\left( \bigl(\xymatrix{ \cdot \ar@{-}[r] &\cdot  } \bigr) \sqcup \bigsqcup_{i=1}^{|V(G)|-2} \cdot \right) \right) 
= \{1,2,\dots,k\},
\end{equation}
for some $k$,
then 
\[
\bigcup_{H \subseteq K_{|V(G)|}}s(I(H)*G) = \{0,\dots,l\}
\]
where $l$ is the only element of $
s(I(K_{|V(G)|})*G).
$
\end{theorem}

\begin{proof}
Let we have a $m\in \bigcup_{H \subseteq K_{|V(G)|}}(s(I(H)*G))$ such that $m \neq l$ we will show that also 
$m+1 \in \bigcup_{H \subseteq K_{|V(G)|}}(s(I(H)*G))$.
Let \[
m = s(I(H)*_f G),
\]
if there is an edge $e \in E(G)$ such that $v_G(e)= 1$ and $v_{I(H)}(f^{-1} (e))=0$ then we will consider a graph
$H'=H\cup f^{-1}(e)$ and now 
\[
m+1=s(I(H')*_f G).
\]
if there is no that edge we will consider a 
\[
i=\min \left\{ v_G(e)  \middle| e\in E(G),\, v_{I(H)}(f^{-1} (e))=0  \right\},
\]
this set is non-empty because $m\neq l$ and by the assumption $i\geq 2$.
Let $e_1$ be such an edge for which the set is minimized. By minimality of $i$ and by property \eqref{huste} there exists an edge $e_2$ such that 
$
v_G(e_2)=i-1
$ and $v_{I(H)}(f^{-1} (e_2))=1$. Let now consider a graph $H'= H \cup e_1 \setminus e_2$, this graph again satisfy the condition what we want
\[
m+1=s(I(H')*_f G).
\]

\end{proof}

\begin{corollary}
Let $G$ be an unweighted connected graph then
\[
\bigcup_{H \subseteq K_{|V(G)|}}\mathscr{H}_H(G) = \{0,\dots,h^+_{K_{|V(G)|}}(G)\},
\]
where $h^+_{H}(G)=\max\{\mathscr{H}_H(G)\}$.
\end{corollary}

\begin{proof}
This corollary is only an application of theorems \ref{hamspec} and \ref{dense}, the only thing which remains is to check the condition
\[
s\left(\rho (G)*I\left( \bigl(\xymatrix{ \cdot \ar@{-}[r] &\cdot  } \bigr) \sqcup \bigsqcup_{i=1}^{|V(G)|-2} \cdot \right) \right) 
= \{1,2,\dots,k\},
\]
but this is obvious, let $v,u\in V(G)$ be vertices with distance $i \geq 2 $, let $v,v_1,\dots,v_{i-1},u$ be a path of minimal length connecting $v$ and $u$ then now the distance between  
$v_1$ and $u$ is $i-1$.
\end{proof}

\begin{theorem}
Let $R$ be a ring then we have an inclusion 
\[
R \hookrightarrow \mathscr{P}(R-Graphs_n) 
\]
given by 
\[
r \mapsto r\{\cdot I(K_n)\},
\]
respecting the structure.
\end{theorem}

\begin{remark}
We would like to say that it is a homomorphism of rings but the right side is not a ring, but if the structure of 
$\mathscr{P}(R-Graphs_n)$ has a name (for example ordered weak ring) it will be a homomorphism of that. 
\end{remark}

\begin{proof}
\begin{align*}
(r\cdot s)\cdot \{I(K_n)\} &= \left( r\cdot \{I(K_n)\} \right)* \left( s\cdot \{I(K_n)\} \right) \\
(r +  s)\cdot \{I(K_n)\} &= \left( r\cdot \{I(K_n)\} \right)+ \left( s\cdot \{I(K_n)\} \right)
\end{align*}
and it also preserves $0$ and $1$ of these two "weak rings". 
\end{proof}

\begin{proposition}
Let $R$ be a ring then $R$ is a domain if and only if \[\mathscr{P}(R-Graphs_n)\] is a "domain" (this means that it does not contain zero divisors).
\end{proposition}

\begin{proof}
Let $R$ be not a domain and $r,s\in R$ such that $ r\cdot s = 0$, then \[
0=0\cdot I(K_n)=(r\cdot s)\cdot I(K_n)=\left( r\cdot I(K_n) \right)* \left( s\cdot I(K_n) \right)
\]
and this means that $\mathscr{P}(R-Graphs_n)$ is not a domain.

Let now $R$ be a domain and \[
\mathscr{R},\mathscr{S}\in \mathscr{P}(R-Graphs_n)\text{ such that } \mathscr{R}*\mathscr{S} = \{ 0 \},
\]
where the $0$ means the graph with all edges weighted by $0$.
We will show that one of them is equal to zero. If both of them have nonzero edge $e,e'$ then in $\mathscr{R}*_f  \mathscr{S}$
exists at least one edge weighted by product of weights of $e,e'$ and because $R$ is a domain, this product is nonzero.
\end{proof}

\begin{theorem}
Let $H,~G$ be unweighted graphs, where $G$ is connected then
\[h_H (P_{|V(G)|-1}) - \binom{|V(G)|+1}{ 3} +h_{K_{|V(G)|}}(G) 
\leq h_H(G),\]
where $h_{H}(G)=\min\{\mathscr{H}_H(G)\}$. 
Moreover, if $\overline{H}$ is connected then both sides are equal if and only if $G$ is a complement of a path.
\end{theorem}

\begin{proof}
The main theorem of the article \cite{cit1} is this:
\[
h^+ (G)\leq h^+_H (P_{|V(G)|-1})
\]
moreover, if $H$ is connected then both sides are equal if and only if $G$ is a path. If we translate it by the theorem \ref{hamspec}
we will get this
\[
\max s(I(H)* \rho(G))\leq \max s(I(H)* \rho(P_{|V(G)|-1})).
\]
Let us substitute there $H=K_{|V(G)|}\setminus H$ and make the following computation.
\begin{align*}
\text{LHS}&\overset{\phantom{T\ref{diststrict}}}{=}\max s\bigl(I(K_{|V(G)|}\setminus H)* \rho(G)\bigr)\\
&\overset{\phantom{T\ref{diststrict}}}{=}\max s\bigl( (I(K_{|V(G)|})+ (-1) \cdot I(H) ) * \rho(G)   \bigr)\\
 	  &\overset{T\ref{diststrict}}{=}\max\bigl(s\bigl(I(K_{|V(G)|}) *\rho(G)\bigr) - s(I(H)* \rho(G))\bigr),\\
	  &\overset{\phantom{T\ref{diststrict}}}{=}s\bigl(I(K_{|V(G)|}) *\rho(G)\bigr) - \min s(I(H)* \rho(G)).
\end{align*}
Theorem \ref{diststrict} can be used because 
\[
|s\left((-1) \cdot I(K_{|V(G)|})* \rho(G)\right)|=1.
\]
Analogously

\begin{align*}
\text{RHS}&\overset{\phantom{T\ref{diststrict}}}{=}\max s\bigl(I(K_{|V(G)|}\setminus H)* \rho(P_{|V(G)|-1})\bigr)\\
	  &\overset{\phantom{T\ref{diststrict}}}{=}s\bigl(I(K_{|V(G)|}) *\rho(P_{|V(G)|-1})\bigr) - \min s\bigl(I(H)* \rho(P_{|V(G)|-1})\bigr).
\end{align*}
\begin{align*}
&\min s\bigl(I(H)* \rho(P_{|V(G)|-1})\bigr) - s\bigl(I(K_{|V(G)|}) *\rho(P_{|V(G)|-1})\bigr) + \\
& \phantom{ssssssssssssssssssssssssssssssssssssssss} +s\bigl(I(K_{|V(G)|}) *\rho(G)) \\
&\leq \min s(I(H)* \rho(G)\bigr). 
\end{align*}
The last thing is that 
\begin{align*}
s&\bigl(I(K_{|V(G)|}) *\rho(P_{|V(G)|-1}) \\
&= \bigl(1+2+\dots + (|V(G)|-1) \bigr) + \bigl(1 +2 + \dots + ( |V(G)|-2)\bigr) +  \dots + \bigl( 1 \bigr)\\
&= \binom{|V(G)|}{2} + \binom{|V(G)|-1}{2}+ \dots + \binom{2}{2} = \binom{|V(G)|+1}{3}. 
\end{align*}

\end{proof}

Now we are going to show some sufficient conditions for $|s(H*G)|=1$, as we know from Theorem \ref{diststrict} and Lemma \ref{l-1}, couples of graphs (or sets of graphs) with this property
satisfy strict distributivity, hence they are good for calculations as we have seen in previous Theorem. We would like to have much more of those couples. For this aim we are going to introduce the following relation on the set of pseudoorderings.

\begin{definition}
Let $G$ and $H$ be $R$-weighted complete graphs and two pseudorderings
\[
g,f:H \rightarrow G.
\]
We will define a relation $\sim$ by the following 
\begin{align*}
f&\sim g, \text{  if and only if  } \\
\exists \varphi \in \Aut{H},\, &\exists\psi \in \Aut{G} 
\text{  such that  }  g=\psi \circ f \circ  \varphi. 
\end{align*}
\end{definition}

\begin{proposition}
The relation $\sim$ is an equivalence relation.
\end{proposition}

\begin{proof}
It is reflective
\[
id \circ f \circ id = f,~ id\in \Aut{H},\, id\in \Aut{G}
\]
It is symmetric, let 
\[
g=\psi \circ f \circ  \varphi,~ \varphi\in \Aut{H},\, \psi\in \Aut{G},
\]
then
\[
f=\varphi^{-1} \circ g \circ  \psi^{-1}, ~ \psi^{-1} \in \Aut{G},\, \varphi^{-1} \in \Aut{H}.
\] 
It is also transitive, let
\[
f\sim g, g\sim h, 
\]
this is equivalent to 
\[
g=\psi \circ f \circ \varphi,~ h= \psi' \circ g \circ \varphi', ~\varphi,\varphi'\in \Aut{H},\, \psi,\psi'\in \Aut{G}
\]
hence 
\[
h= \psi' \circ \psi \circ f \circ \varphi \circ \varphi',~ \psi' \circ \psi\in \Aut{G},\, \varphi \circ \varphi'\in  \Aut{H}.
\]
\end{proof}

\begin{lemma}\label{l13}
\[
g(H)*_f G \overset{\phantom{bb}g^{-1}}{\cong} H*_{f\circ g} G = H*_g f^{-1}  (G). 
\]
\end{lemma}

\begin{proof}
Lets prove the second part $H*_{f\circ g} G = H*_g f^{-1}  (G)$.
\[
v_{H*_{f\circ g} G}(e)=v_H(e) \cdot v_G (f(g(e))) 
\]
\[
v_{H*_g f^{-1}  (G)}(e)=v_H(e) \cdot v_{f^{-1}  (G)}(g(e)) = v_H(e) \cdot v_{G}(f(g(e))).
\]
For the first part we need to show that 
\[
v_{g(H)*_f G}(g(e)) = H*_{f\circ g} G,
\]
\[
v_{g(H)*_f G}(g(e)) = v_{g(H)}(g(e)) \cdot v_G (f(g(e)))= v_{H}(e) \cdot v_G (f(g(e))). 
\]
\end{proof}

\begin{definition}
Let $H$ and $G$ be $R$-weighted complete graphs (or unweighted graphs) 
we denote a set of all pseudoorderings as follows
\[
pso(H,G)=\{f: V(H)\rightarrow V(G)|f \text{ bijection}\}.
\]
\end{definition}

\begin{proposition}\label{p1}
Let $H,~G$ be $R$-weighted complete graphs then the map \[
H*_{-} G: pso(H,G)\rightarrow H*G
\]
converts $\sim$ in to an isomorphism.
\end{proposition}

\begin{proof}
Let 
\[
g=\psi \circ f \circ  \varphi,~ \varphi\in \Aut{H},\, \psi\in \Aut{G},
\]
then
\[
H*_{\psi \circ f \circ  \varphi} G \overset{L \ref{l13}}{=} H*_{ f \circ  \varphi}  \psi^{-1}( G) \overset{\psi\in \Aut{G}}{=}
H*_{ f \circ  \varphi} G \overset{\substack{ L \ref{l13}\\ \varphi}}{\cong } \varphi(H)*_f G \overset{\varphi\in \Aut{H}}{=} 
H*_f G.
\]
\end{proof}

\begin{theorem}\label{taut}
Let $H,~G$ be $R$-weighted complete graphs then
\[
\left|\sfrac{pso(H,G)}{\sim} \right|=1 \implies |H\*G|=1.
\]
\end{theorem}

\begin{proof}
If any two pseudoorderings $f$ and $g$ are in relation $f\sim g$ then by the proposition \ref{p1} are get isomorphic elements of combinatorial spectrum
\[
H*_f G \cong H*_g G,
\] 
hence if all are isomorphic then
\[
|H\*G|=1.
\]
\end{proof}

\begin{corollary}
Let $H,~G$ be $R$-weighted complete graphs then
\[
\left|\sfrac{pso(H,G)}{\sim} \right|=1 \implies |s(H*G)|=1.
\]
\end{corollary}

\begin{corollary}\label{c1}
Let $H,~G$ be $R$-weighted complete graphs with the same set of vertices then
\[
\Aut{G} \circ \Aut{H}= pso(H,G) \implies |H\*G|=1 \implies |s(H*G)|=1.
\]
\end{corollary}

\begin{remark}
The condition of having the same set of vertices is only for a simplification of the statement we know that a combinatorial spectrum 
will not changed if we replace a graph by some isomorphic one. This means that we can all our graphs consider with the set of vertices $\{1,\dots,|V(G)|\}$.
\end{remark}

\begin{proof}
We will show that  
\[
\Aut{G} \circ \Aut{H}= pso(H,G) \implies \left|\sfrac{pso(H,G)}{\sim} \right|=1,
\]
then it is enough to use the theorem \ref{taut}.
Let $f\in pso(H,G)$, by the condition above there exist $\psi \in \Aut{G},\,\varphi \in \Aut{H}$, such that 
\[
\psi \circ \varphi = f.
\] 
This we can rewrite as 
\[
f=\psi \circ id \circ \varphi
\]
and this is by the definition $id \sim f$. And now if all pseudoorderings are in $\sim $ with $id$ then  
\[
\left|\sfrac{pso(H,G)}{\sim} \right|=1.
\]
\end{proof}

\begin{remark}
Assuming the same set of vertices, it is even an equivalence
\[
\Aut{G} \circ \Aut{H}= pso(H,G) \iff \left|\sfrac{pso(H,G)}{\sim} \right|=1.
\]
It remains to show the implication from right to left.
Let $f\in  pso(H,G)$, it is in relation $id\sim f$ this means that 
\[
f=\psi \circ id \circ  \varphi=\psi \circ  \varphi ,~ \varphi\in \Aut{H},\, \psi\in \Aut{G},
\]
hence $f\in \Aut{G} \circ \Aut{H}$.

It also holds 
\[
\Aut{G} \circ \Aut{H}= pso(H,G) \iff \Aut{H} \circ \Aut{G}= pso(H,G).
\]
This can be proven by applying $\phantom{a}^{-1}$,
\begin{align*}
(\Aut{G} \circ \Aut{H})^{-1}&= \left(pso(H,G) \right)^{-1}\\
\Aut{H}^{-1} \circ \Aut{G}^{-1}&= \left(pso(H,G) \right)^{-1}\\
\Aut{G} \circ \Aut{H}&= pso(H,G).
\end{align*}
\end{remark}

\begin{lemma}
Let $G$ be an unweighted graph then 
\[
\Aut{G}=\Aut{I(G)}=\Aut{\rho(G)}.
\]
\end{lemma}

\begin{proof}
First equation is obvious. Let $\varphi \in \Aut{G}$ it is also an
isomorphism of $\rho(G)$ because isomorphisms preserves distances.

Let $\varphi \in \Aut{\rho(G)}$ because 
\[
U_1(\rho(G))= G,
\]
hence it is also an isomorphism of $G$.
\end{proof}

\begin{remark}
The opposite implication from the corollary \ref{c1} does not hold,
\[
|s(H*G)|=1 \notimplies \Aut{G} \circ \Aut{H}= pso(H,G).
\]
Let us denote the following graph by $H'$.
\begin{center}
\definecolor{qqqqff}{rgb}{0,0,1}
\definecolor{uququq}{rgb}{0.25,0.25,0.25}
\definecolor{zzttqq}{rgb}{0.6,0.2,0}
\definecolor{xdxdff}{rgb}{0.49,0.49,1}
\begin{tikzpicture}[line cap=round,line join=round,>=triangle 45,x=1.0cm,y=1.0cm]
\clip(-0.56,-3.34) rectangle (8.25,3.74);
\fill[color=zzttqq,fill=zzttqq,fill opacity=0.1] (0,1) -- (0,-1) -- (1.9,-1.62) -- (3.08,0) -- (1.9,1.62) -- cycle;
\fill[color=zzttqq,fill=zzttqq,fill opacity=0.1] (8,-1) -- (8,1) -- (6.1,1.62) -- (4.92,0) -- (6.1,-1.62) -- cycle;
\draw [color=zzttqq] (0,1)-- (0,-1);
\draw [color=zzttqq] (0,-1)-- (1.9,-1.62);
\draw [color=zzttqq] (1.9,-1.62)-- (3.08,0);
\draw [color=zzttqq] (3.08,0)-- (1.9,1.62);
\draw [color=zzttqq] (1.9,1.62)-- (0,1);
\draw [color=zzttqq] (8,-1)-- (8,1);
\draw [color=zzttqq] (8,1)-- (6.1,1.62);
\draw [color=zzttqq] (6.1,1.62)-- (4.92,0);
\draw [color=zzttqq] (4.92,0)-- (6.1,-1.62);
\draw [color=zzttqq] (6.1,-1.62)-- (8,-1);
\draw (3.08,0)-- (4.92,0);
\draw (3.08,0)-- (4.92,0);
\draw (1.9,1.62)-- (0,-1);
\draw (0,1)-- (1.9,-1.62);
\draw (6.1,1.62)-- (8,-1);
\draw (8,1)-- (6.1,-1.62);
\begin{scriptsize}
\fill [color=xdxdff] (0,1) circle (1.5pt);
\fill [color=xdxdff] (0,-1) circle (1.5pt);
\fill [color=uququq] (1.9,-1.62) circle (1.5pt);
\fill [color=uququq] (3.08,0) circle (1.5pt);
\fill [color=uququq] (1.9,1.62) circle (1.5pt);
\fill [color=qqqqff] (8,1) circle (1.5pt);
\fill [color=qqqqff] (8,-1) circle (1.5pt);
\fill [color=uququq] (6.1,1.62) circle (1.5pt);
\fill [color=uququq] (4.92,0) circle (1.5pt);
\fill [color=uququq] (6.1,-1.62) circle (1.5pt);
\end{scriptsize}
\end{tikzpicture}
\end{center}
And $G'=K_{9,1}$,
we will prove that $H=I(H'),~ G=\rho(G')$ is a counterexample. 
\begin{align*}
s(I(H')*\rho(G'))&=s(I(H')*(2I(K_{10})-I(G'))\\
&=2s(I(H')*I(K_{10}))-s(I(H')*I(G'))\\
&\overset{\ref{thmdeg}}{=}2\{|E(H')|\}-\{deg_H(v)|v\in V(H)\}\\
&=\{30\}-\{3\}
=\{27\},
\end{align*} if 
\[\Aut{\rho(G)} \circ \Aut{I(H)}= pso(H,G)
\]
was satisfied then it would be satisfied also
\[
pso(H,G)= \Aut{I(H)} \circ   \Aut{\rho(G)}  = \Aut{\rho(H)} \circ   \Aut{I(G)}
\]
and then would be by corollary \ref{c1}
\[
|s(\rho(H)*I(G))|=1.
\] 
But this is not true, if we consider two pseudoorderings, $f$ sending a vertex on the side to a center and $g$ sending some of the vertices on the bridge to a center,
we will get that $s(\rho(H)*_f I(G))>s(\rho(H)*_g I(G))$. 
\end{remark}

\begin{definition}
Let $H,\,G$ be $R$-weighted complete graphs we define 
\[
H \perp G \text{  if and only if  }  |s(H*G)|=1.
\]
\end{definition}

\begin{definition}
Let $\mathscr{G},\,\mathscr{H}$ be classes of $R$-weighted complete graphs we define 
\[
\mathscr{H} \perp \mathscr{G} \text{  if and only if  }  \forall H \in \mathscr{H}, \, \forall G \in \mathscr{G}: ~ H\perp G.
\]
\end{definition}

\begin{definition}
Let $\mathscr{H}$ be a class of $R$-weighted complete graphs we define a class
\[
{\mathscr{H}}^{\perp}= \{ G \,|\,\forall H \in \mathscr{H} :~ H\perp G \}.
\]
Let $H$ be a $R$-weighted complete graph we denote
\[
H^{\perp}=\{H\}^{\perp}.
\]
\end{definition}

\begin{proposition}
\[
I^{-1}\left(I(S_n)^{\perp}\right)=\{\text{set of all regular graphs of order }n\}
\]
where $S_n=K_{1,n-1}$.
\end{proposition}

\begin{proof}
\[
I^{-1}\left(I(S_n)^{\perp}\right)=\{ H \,| I(H)\perp I(S_n) \},
\]
from the theorem \ref{thmdeg} we know that $s(I(H)*I(S_n))$ is a set of degrees of $H$, if it must 
be one element set that means it is a regular graph.
\end{proof}

\begin{proposition}\label{rhoi}
\[
I^{-1}\left(\rho(S_n)^{\perp}\right)=\{\text{set of all regular graphs of order }n\}.
\]
\end{proposition}
\begin{proof}
We can prove it in similar way as before or we can do it better.
We will prove that 
\[
I(S_n)^{\perp}= \rho(S_n)^{\perp}.
\] 
All distances in $S_n$ are $1$ or $2$, this leads us to a formula 
\[
\rho(S_n)= 2I(K_n) - I(S_n),
\]
now by the theorem \ref{diststrict} is \begin{align}
s(H*\rho(S_n))&=\underbrace{2s(H*I(K_n))}_{\text{one number}} ~+~ s((-1)H*I(S_n))\notag\\
s(H*\rho(S_n))&=\underbrace{2s(H*I(K_n))}_{\text{one number}} ~-~ s(H*I(S_n))\notag
\end{align}
hence
\[
I(S_n)^{\perp}= \rho(S_n)^{\perp}.
\]
\end{proof}
From the previous proof we can get a proposition

\begin{proposition}
Let $G$ be an unweighted graph with diameter at most $2$ then 
\[
I(G)^{\perp}= \rho(G)^{\perp}.
\]
\end{proposition}

\begin{definition}
Let $\mathscr{H},\,\mathscr{G}$ be classes of $R$-weighted complete graphs and $\mathscr{A}\subseteq \mathscr{P}(R)$, we define 
\[
\mathscr{H} \perp_{\mathscr{A}} \mathscr{G} \text{  if and only if  }  \forall H \in \mathscr{H}, \, \forall G \in \mathscr{G}: ~ s(H*G)\in \mathscr{A}.
\]
Analogously we define a set
\[
{\mathscr{H}}^{\perp_{\mathscr{A}}}= \{ G \,|\,\forall H \in \mathscr{H} :~ s(H*G)\in \mathscr{A} \}.
\]
\end{definition}

\begin{theorem}\label{uzavretost}
Let $\mathscr{A}\subseteq \mathscr{P}(R)$ be a set closed under $+$ and under multiplication by an element of $R$ and it is closed under subsets (this we will call an ideal), then
\[ 
{\mathscr{H}}^{\perp_{\mathscr{A}}}
\]
is closed under $+$, multiplication by an element of $R$.
\end{theorem}

\begin{proof}
Let $G,G'\in {\mathscr{H}}^{\perp_{\mathscr{A}}}$, this means that 
\[
\forall H \in \mathscr{H} :~ s(H*G),s(H*G')\in \mathscr{A}. 
\]
By the theorem \ref{Tad}
\[
s(H*(G+G'))\subseteq s(H*G) + s(H*G'), 
\]
because $s(H*G),s(H*G')\in \mathscr{A}$ then by the assumptions also $s(H*G) + s(H*G')\in \mathscr{A}$ and then also $s(H*(G+G'))\in \mathscr{A}$. 

Let now $G\in {\mathscr{H}}^{\perp_{\mathscr{A}}}$ and $r\in R$, this means that
\[
\forall H \in \mathscr{H} :~ s(H*G)\in \mathscr{A},
\] 
hence
\[
\forall H \in \mathscr{H} :~ s(H*(r\cdot G))\in \mathscr{A} \iff \forall H \in \mathscr{H} :~ r\cdot s(H*G)\in \mathscr{A}
\]
and the right hand side is true by assumptions.
\end{proof}

\begin{proposition}
\[
I^{-1}\left(\rho(\{\text{connected bipartite graphs}\})^{\perp_{\mathscr{P}(2\mathbb{Z})}} \right)=\{\text{GWEC}\},
\]
Where GWEC means graphs with Eulerian components. 
\end{proposition}

\begin{proof}
At first we will prove that for $n\in \mathbb{N},\, n\leq |V(G)|$ is satisfied 
\[
s\left(I\left(C_{n}\sqcup \bigsqcup_{i=1}^{|V(G)|-n} \cdot \right) * \rho(G)\right) \in \mathscr{P}(2\mathbb{Z}),
\]
for $G$ connected bipartite graph. Let $A,\overline{A} \subseteq V(G)$ such that all edges of $G$ are between $A$ and $\overline{A}$.
All distances between these two sets are odd and all distances inside these two sets are even. 
\[
s\left(I\left(C_{n}\sqcup \bigsqcup_{i=1}^{|V(G)|-n} \cdot \right)*_f \rho(G)\right)=\sum_{i=1}^{n_a} a_i + \sum_{j=1}^{n_b} b_j
\] 
where $a_i$ is even and $b_j$ is odd and $n_b$ is even.
Indeed, $n_b$ is a number of edges \[
\{u,v\}\in E \left(C_{n}\sqcup \bigsqcup_{i=1}^{|V(G)|-n} \cdot \right)
\]
such that $|\{f(u),f(v)\}\cap A|=1$ and this needs to be even because as many times 
a cycle \[
f\left(\left(C_{n}\sqcup \bigsqcup_{i=1}^{|V(G)|-n} \cdot \right)\right) \]
leaves $A$ as many times it needs to come back.

The set $\mathscr{P}(2\mathbb{Z})$ is an ideal, by the theorem \ref{uzavretost}
is \[\rho(\{\text{connected bipartite graphs}\})^{\perp_{\mathscr{P}(2\mathbb{Z})}}\] 
closed under $+$,
hence
\[
I^{-1}\left(\rho(\{\text{connected bipartite graphs}\})^{\perp_{\mathscr{P}(2\mathbb{Z})}} \right)
\]
is closed under disjoint unions. We have proven that 
\[
\left(C_{n}\sqcup \bigsqcup_{i=1}^{|V(G)|-n} \cdot\right) \in I^{-1}\left(\rho(\{\text{connected bipartite graphs}\})^{\perp_{\mathscr{P}(2\mathbb{Z})}} \right),
\]
then also all disjoint unions of these graphs are also in this set. Any graph with Eulerian components can be written as disjoint union of graphs of the shape 
\[
C_{n}\sqcup \bigsqcup_{i=1}^{|V(G)|-n} \cdot,
\]
hence 
\[
\{\text{GWEC}\}\subseteq I^{-1}\left(\rho(\{\text{connected bipartite graphs}\})^{\perp_{\mathscr{P}(2\mathbb{Z})}} \right).
\]

Let now \[
G\in I^{-1}\left(\rho(\{\text{connected bipartite graphs}\})^{\perp_{\mathscr{P}(2\mathbb{Z})}} \right)
\]
this at least means that 
\[
s\left(I(G)*\rho(K_{1,|V(G)|-1})\right)\in \mathscr{P}(2\mathbb{Z}).
\]
We know that 
\[
\rho\left(K_{1,|V(G)|-1}\right)=2\cdot I\left(K_{|V(G)|}\right)-I\left(K_{1,|V(G)|-1}\right)
\]
it is explained in the proof of the proposition \ref{rhoi}. 
\begin{align*}
s(
I(G) &*\rho
(K_{1,|V(G)|-1}))\\
& =s\left(I(G)*\left(2\cdot I\left(K_{|V(G)|}\right)-I\left(K_{1,|V(G)|-1}\right)\right)\right)\\
& =2s\left(I(G)*I\left(K_{|V(G)|}\right)\right)-s\left(I(G)*I\left(K_{1,|V(G)|-1}\right)\right),
\end{align*}
hence 
\[
s\left(I(G)*\rho(K_{1,|V(G)|-1})\right)\in \mathscr{P}(2\mathbb{Z}) \iff s\left(I(G)*I(K_{1,|V(G)|-1})\right)\in \mathscr{P}(2\mathbb{Z}).
\]
By the theorem \ref{thmdeg} we know that 
$
s\left(I(G)*I(K_{1,|V(G)|-1})\right)$ is equal to set of degrees of $G$, hence we get that all degrees of $G$ are even and this means that $G$
has Eulerian components. And this gives us the second inequality
\[
\{\text{GWEC}\}\supseteq I^{-1}\left(\rho(\{\text{connected bipartite graphs}\})^{\perp_{\mathscr{P}(2\mathbb{Z})}} \right).
\]
\end{proof}

\renewcommand{\bibname}{references}

\end{document}